\documentclass{amsart}

\title{Quasilinearization with regularizing tensor paraproducts}
\author{oluwadamilola fasina}
\date{December 2025}
\usepackage{amssymb,amsthm,amsmath,mathtools,float,comment,soul,xcolor,verbatim,graphicx}
\usepackage{hyperref}
\usepackage{cleveref}

\newtheorem{theorem}{Theorem}[section]
\newtheorem{proposition}[theorem]{Proposition}
\newtheorem{lemma}[theorem]{Lemma}
\newtheorem{corollary}[theorem]{Corollary}
\newtheorem{definition}[theorem]{Definition}
\newtheorem{remark}[theorem]{Remark}
\newtheorem{example}[theorem]{Example}

\begin{document}

\maketitle

\begin{abstract}
We extend Bony's celebrated work on paraproducts to continous and multiscale \emph{tensor} paraproducts. For $A \in \mathcal{C}^2(\mathbb{R})$ and $f \in \Lambda_{\alpha}([0,1]^2, d_d(x,y)^{\alpha} \times d'_d(x',y')^{\alpha})$, we construct an approximation, $\tilde{A}_{(N,N')}(f)$ to $A(f)$, replacing the operator $T: f \to A(f)$ with the continous tensor paraproduct, $\Pi^{(t,t')}_{(A',A'')}$, and the multiscale tensor paraproduct $\Pi^{(N,N')}_{(A',A'')}:f \to \tilde{A}_{(N,N')}(f) + \Delta_{ (N,N')}(A,f)$. In the multiscale case, we provide estimates on the residual, $\Delta_{(N,N')}(A,f)$, and show it has twice the regularity of $f$ such that
$\Delta_{(N,N')}(A,f) \in \Lambda_{2 \alpha}([0,1]^2)$ and $\lVert \Delta_{(N,N')}(A,f) \rVert_{\Lambda_{2\alpha}([0,1]^2)} \leq C_A \lVert f \rVert_{\Lambda_{\alpha}([0,1]^2)} $. Our theoretical findings are supplemented with a computational example. \\

\begin{comment}
\begin{enumerate}
    \item We replace the operator $T(f)$ with the quasilinear "paradifferential operator" $T_{N,N'}$
    \item Just say $A(f)$ is Holder and the residual is $2\alpha$-Holder
\end{enumerate}
\end{comment}

\end{abstract}

\section{Introduction}

\begin{comment}
\begin{enumerate}
    \item Brief mention of Bony's quasilinearization formula
    \item Contemporary work on paraproducts/what we're doing (mixed holder regularity w.r.t to dyadic distances)
    \item mention relevance to multiscale data processing/pde
\end{enumerate}
\end{comment}

Bony's seminal work on paraproducts \cite{bony1981calcul} introduced a technique for separating the frequency components of operators, facilitating simpler operator and nonlinear PDE analysis. We extend Bony's work on paraproducts with a program concerning continuous and multiscale \emph{tensor} paraproducts for the operator:

\begin{align}
T : f \to A(f)
\label{eq:1}
\end{align}

where $f \in \Lambda_{\alpha}([0,1]^2, d_d(x,y)^{\alpha} \times d'_d(x',y')^{\alpha} , A \in \mathcal{C}^2(\mathbb{R})$.  $\Lambda_{\alpha}([0,1]^2, d_d(x,y)^{\alpha} \times d'_d(x',y')^{\alpha})$ is the space of mixed-Holder functions with respect to the tensor product of dyadic distances on each axis, and $A$ acts nonlinearly on $f$. In particular we replace the operator in (\ref{eq:1}) with the continuous (in the sense of scaling parameters) tensor paraproduct, $\Pi^{(t,t')}_{(A',A'')}$, and the multsicale tensor paraproduct, $\Pi^{(N,N')}_{(A',A'')}$. \\

Contemporary work on paraproducts is typically concerned with the regularity of bilinear and multilinear operators. In \cite{muscalu2004bi,muscalu2006multi}, regularity results of bilinear and multilinear operators are extended from one parameter to biparameters and multiple parameters, encoding the order of the pseudo-differential operator into analysis; i.e. estimates on the symbols $m(\xi_1,\eta_1) \times m(\xi_2,\eta_2)$ in the bilinear, biparameter case or the symbols $m_1(\xi_1,\xi_2, \ldots, \xi_n) \times m_2(\xi_1,\xi_2, \ldots, \xi_n) \times \ldots \times m_k(\xi_1,\xi_2, \ldots, \xi_n) $ in the k-linear, n-parameter case yield regularity results resembling those produced by Coifman and Meyer in \cite{meyer1990ondelettes}. Other examples are \cite{bonami2012paraproducts} who use paraproducts to split the bilinear operator $\mathcal{H}^1(\mathbb{R}^n) \times \mathcal{BMO}(\mathbb{R}^n) \to L^1(\mathbb{R}^n) + \mathcal{H}^{\phi}_{\omega}(\mathbb{R}^n)$ into the sum of operators $\mathcal{H}^1(\mathbb{R}^n) \times \mathcal{BMO}(\mathbb{R}^n) \to L^1(\mathbb{R}^n)$  and  $\mathcal{H}^1(\mathbb{R}^n) \times \mathcal{BMO}(\mathbb{R}^n) \to \mathcal{H}^{log}(\mathbb{R}^n)$ (where $\mathcal{H}^{\phi}_{\omega}(\mathbb{R}^n)$ is a Hardy-Orlicz space) and \cite{chang2022boundedness} who construct operator estimates for paraproducts on homogeneous spaces introduced by Coifman and Weiss \cite{coifman2006analyse}.\\

While construction of paraproducts from \cite{muscalu2004bi,muscalu2006multi} are for biparameter bilinear $T^{(2)}_m(f,g)$ and multilinear $T^{(2)}_m(f_1 \times f_2 \times \ldots \times f_n)$ operators and multiparameter bilinear $T^{(d)}_m(f,g)$ and multilinear $T^{(d)}_m(f_1 \times f_2 \times \ldots \times f_n)$ operators, which decompose operators based on their order, we consider an operator decomposition based on the dimension of the support of the function it acts on (in this case $d=2$ for $[0,1]^2$), permitting analytical control along each dimension. The main result we prove is:

\begin{theorem}
Suppose  $A \in \mathcal{C}^2(\mathbb{R})$, $ f \in  \Lambda_\alpha([0,1]^2), 0 < \alpha < \frac{1}{2}$, then for the operator $T: f \to A(f)$ we can approximate $A(f)$ with

\begin{align}
&  \tilde{A}_{(N,N')}(f) =  \sum_{j=0}^{N} \sum_{j'=0}^{N'}  A'(P^jP'^{j'}(f)) Q^jQ'^{j'}(f) + A''(P^jP'^{j'}(f))Q^jP'^{j'}(f)P^jQ'^{j'}(f)  
\label{eq:2}
\end{align}

such that the multiscale tensor paraproduct transforms $T : f \to A(f)$ to 

\begin{align}
\Pi^{(N,N')}_{(A',A'')} : f \to \tilde{A}_{(N,N')}(f) + \Delta_{(N,N')}(A,f)
\label{eq:3}
\end{align}

where $\Delta_{(N,N')}(A,f) = A(f) - \tilde{A}_{(N,N')}(f) \in \Lambda_{2\alpha}([0,1]^2)$ is the residual which has twice the regularity of $f$ and 

\begin{align}
\lVert \Delta_{(N,N')}(A,f) \rVert_{\Lambda_{2\alpha}([0,1]^2)} \leq C_A \lVert f \rVert_{\Lambda_{\alpha}([0,1]^2)}
\label{eq:4}
\end{align}
\label{thm:1}
\end{theorem}

To prove this result we introduce a novel lemma (proved in section 3) which connects the mixed-Holder regularity of $f$ to the decay of expansion coefficients of $f$ in an orthonormal tensor basis:

\begin{lemma}
Suppose $\beta^{j,j'}_{k,k'} $ is an expansion coefficient of $f(x,x')$ in an orthonormal tensor basis supported on the dyadic rectangle $I^j_k \times I^{j'}_{k'}$. One obtains $ |\beta^{j,j'}_{k,k'}| \leq 2^{-(j+j')(2\alpha + \epsilon)}, \epsilon \geq \frac{1}{2} \hspace{0.2 cm} \forall \hspace{0.2 cm} I^j_k \times I^{j'}_{k'} \in [0,1]^2$ \text{iff} $f(x,x') \in \Lambda_{2\alpha}([0,1]^2)$
\label{lemma:1.2}
\end{lemma}

We are motivated by problems in matrix inference where the mixed-Holder regularity appears naturally as the rows and columns of the matrix have different structure and regularity \cite{gavish2010multiscale,gavish2012sampling} and can also be organized to have mixed-Holder regularity \cite{ankenman2014geometry}. Theorem \ref{thm:1} shows we can replace the operator $T: f \to A(f)$ with the multiscale tensor paraproduct $\Pi^{(N,N')}_{(A',A'')} : f \to \tilde{A}_{(N,N')}(f) + \Delta_{(N,N')}(A,f)$, such that $A$ is quasilinearized by replacing it with the high-low paraproducts, $A'(P^jP'^{j'}),A''(P^jP'^{j'})$ constructed in the approximation, and has a residual, $\Delta_{(N,N')}(A,f)$ which is twice as smooth as the input data, $f$. $\Pi^{(N,N')}_{(A',A'')}$ permits one to separate the singularites from the smooth parts of data, which has both theoretical and empirical implications in data processing, matrix inference, and numerical analysis. \\

\textbf{Paper Outline:} In Section 3 we provide key preliminaries and definitions needed to build the approximation and prove quantiative estimates on the residual. In section 4 we provide an explicit construction of the continous and multiscale \emph{tensor} paraproducts and prove the first part of the main theorem. In section 5 we complete the proof the main theorem by establishing quantitative estimates on the residual. We conclude by illustrating the utility of the multsicale tensor paraproduct with a computational example.

\section{Acknowledgments}

The author would like to thank Ronald R. Coifman for bringing the problem to the authors' attention, for numerous enlightening discussions, and his unique perspective on the problem.

\section{Preliminaries: Key Definitions and Lemmas}

\begin{comment}
\begin{enumerate}
    \item Holder regularity w.r.t dyadic distance
    \item Wavelets
    \item Convolution Operators
    \item Lemmas concerning Holder regularity and coefficient decay bidirectional implication
    \item Definitions establishing connection between Holder norm and wavelet coefficients
\end{enumerate}
\end{comment}

\begin{definition}
Let the dyadic distance along an axis be:
    
\begin{align}
d_d(x,y) = \inf_{\tilde{I}} \{ |I_k^j| : (x,y) \in I_k^j \} 
\label{eq:5}
\end{align}
where $\tilde{I} = \{ |I_k^j | : I_k^j = [2^{-j}k,2^{-j}(k+1)), \forall j,k \in \mathbb{N} \}$ and $I^j_k$ is a dyadic interval along direction $j$ and $I^{j'}_{k'}$ is a dyadic interval along direction $j'$
\label{defn:3.1}
\end{definition}

Thus for the set $ [0,1]^2 $ one can consider the dyadic distances $d_d(x,y), d'_d(x',y')$ along each axis. Now we define mixed $\alpha$-Holder and $\alpha$-Holder regularity with respect to dyadic distances.

\begin{definition}
 $f: (x,x') \to \mathbb{R} \in \Lambda_{\alpha}([0,1]^2,d_d(x,y)^{\alpha} \times d'_d(x',y')^{\alpha})$ for $(x,x') \in [0,1]^2$ if the following conditions  are satisfied:

\begin{align}
\frac{ | f(x,x') - f(y,x') - f(x,y') + f(y,y') | } {d_d(x,y)^{\alpha}d'_d(x',y')^{\alpha}} \leq C,  \hspace{0.2 cm} \forall \hspace{0.2 cm} (x,x'), (y,x'), (x,y'), (y,y')  \hspace{0.2 cm} \in [0,1]^2
\label{eq:6}
\end{align}  

\begin{align}
\frac{ | f(x,x') - f(y,x') | }{d_d(x,y)^{\alpha}} \leq C, \hspace{0.2 cm} \forall \hspace{0.2 cm} (x,x'), (y,x') \in  [0,1]^2
\label{eq:7}
\end{align}

\begin{align}
\frac{ | f(x,x') - f(x,y') |}{d'_d(x',y')^{\alpha}} \leq C, \hspace{0.2 cm} \forall \hspace{0.2 cm} (x,x'), (x,y') \in  [0,1]^2
\label{eq:8}
\end{align}

where $\alpha,C > 0$ and $(x,x'), (y,x') ,(x,y'), (y,y')$ are vertices of a rectangle in $[0,1]^2$. We refer to $\Lambda_{\alpha}([0,1]^2,d_d(x,y)^{\alpha} \times d'_d(x',y')^{\alpha})$ as the space of mixed $\alpha$-holder functions with respect to the dyadic distance. As we only consider mixed $\alpha$-Holder functions with respect to dyadic distances we abbreviate $\Lambda_{\alpha}([0,1]^2,d_d(x,y)^{\alpha} \times d'_d(x',y')^{\alpha})$ as $\Lambda_{\alpha}([0,1]^2)$ for the rest of the paper. Note the second and third conditions are the usual definitions of $\alpha$-Holder regularity along each axis of $[0,1]^2$.
\label{defn:3.2}
\end{definition}

The norm in $\Lambda_{\alpha}([0,1])$ will be used to quantify the claim that the residual, $\Delta(A,f)$ has twice the regularity of $f$. Different characterizations of Holder norms were studied extensively in \cite{leeb2016holder}. We use the mixed-Holder norm defined in terms of the tensor wavelet coefficients:

\begin{definition} The mixed-Holder norm as measured by the tensor wavelet coefficients:

\begin{align}
\lVert f \rVert_{\Lambda_{\alpha}([0,1]^2)} = \sup_{j,k,j',k'} \frac{| <f(x,x'), \psi^j_k(x) \times \psi^{j'}_{k'}(x')> |}{2^{-(j+j')(\alpha + \frac{1}{2})}}
\label{eq:9}
\end{align}

\label{defn:3.3}
\end{definition}

This norm is useful for providing quantitative estimates of the residual and the computational example in the last section.

\begin{remark}
$\alpha$ can be an arbitrary positive value when the metric is a tree metric (e.g. dyadic distance metric) or tensor product of tree metric, otherwise we restrict $\alpha$ to be  $0  < \alpha < \frac{1}{2}$
\label{rem:3.4}
\end{remark}

To build the multiscale tensor paraproduct, $\Pi^{(N,N')}_{(A',A'')}$, we construct multiscale convolution operators using tensor wavelets. Wavelets form an orthonormal basis for $L^2(\mathbb{R})$ and posses additional desirable properties outlined in \cite{mallat1989theory}. Let $\phi$ and $\psi$ denote the father and mother wavelets as usual. Then the scaling $\phi^j_k(x)$ and wavelet $\psi^j_k(x)$ functions at scale $j=1, \ldots, N$ and location $k=1, \ldots, 2^j$ along one direction of $[0,1]^2$ are defined as:

\begin{definition} 1D wavelet family
\begin{align}
\phi^j_k(x) = 2^{\frac{j}{2}} \phi(2^jx - k), \psi^j_k(x) = 2^{\frac{j}{2}} \psi(2^jx - k) j, k \in \mathbb{N}
\label{eq:10}
\end{align}
\label{def:3.5}
\end{definition}

Let $\phi^{j'}_{k'}(x')$ and $\psi^{j'}_{k'}(x')$ be the scaling and wavelet functions along the other direction of $[0,1]^2$. With definition \ref{def:3.5} we can construct different combinations of orthonormal tensor scaling and wavelet functions: $\phi^j_k(x) \times \phi^{j'}_{k'}(x'), \psi^j_k(x) \times \psi^{j'}_{k'}(x'), \psi^j_k(x) \times \phi^{j'}_{k'}(x'), \psi^j_k(x) \times \psi^{j'}_{k'}(x')$. We now describe several convolution operators which are defined with respect to the different combinations of wavelets and scaling functions:

\begin{definition}
The multiscale convolution operators, $P^{j}P'^{j'},P^{j}Q'^{j'},Q^{j}P'^{j'},Q^{j}Q'^{j'}$ are constructed as follows:
    \begin{itemize}
     \item We call $P^{j}P'^{j'}$ the tensor \textbf{scaling} operator:
    \begin{align}
    P^{j}P'^{j'}(f) \coloneq \sum_{k=1}^{2^j}  \sum_{k'=1'}^{2^{j'}} P^{j}_k P'^{j'}_{k'}(f) =  \sum_{k=1}^{2^j}  \sum_{k'=1'}^{2^{j'}}  \int_{I^j_k \times I^{j'}_{k'}} f(y,y')  (\phi^{j}_k(y) \times \phi^{j'}_{k'}(y')) dy dy' \chi_{I^j_k \times I^{j'}_{k'}}(x,x')
    \label{eq:11}
    \end{align}
    \item We call $P^{j}Q'^{j'}$ the tensor \textbf{scaling-wavelet} operator:
    \begin{align}
    P^{j}Q'^{j'}(f) \coloneq \sum_{k=1}^{2^j}  \sum_{k'=1'}^{2^{j'}} P^{j}_k Q'^{j'}_{k'}(f) =  \sum_{k=1}^{2^j}  \sum_{k'=1'}^{2^{j'}} \int_{I^j_k \times I^{j'}_{k'}} f(y,y') (\phi^{j}_k(y) \times \psi^{j'}_{k'}(y')) dy dy'  \chi_{I^j_k \times I^{j'}_{k'}}(x,x')     
    \label{eq:12}
    \end{align}
    \item We call $Q^{j}P'^{j'}$ the tensor \textbf{wavelet-scaling} operator:
    \begin{align}
    Q^{j}P'^{j'}(f) \coloneq \sum_{k=1}^{2^j}  \sum_{k'=1'}^{2^{j'}} Q^{j}_k P'^{j'}_{k'}(f) =  \sum_{k=1}^{2^j}  \sum_{k'=1'}^{2^{j'}} \int_{I^j_k \times I^{j'}_{k'}} f(y,y') (\psi^{j}_k(y) \times \phi^{j'}_{k'}(y')) dy dy'   \chi_{I^j_k \times I^{j'}_{k'}}(x,x')  \nonumber \\
    \label{eq:13}
    \end{align}
    \item We call $Q^{j}Q'^{j'}$ the tensor \textbf{wavelet} operator
    \begin{align}
    Q^{j}Q'^{j'}(f) \coloneq \sum_{k=1}^{2^j}  \sum_{k'=1'}^{2^{j'}} Q^{j}_k Q'^{j'}_{k'}(f) =  \sum_{k=1}^{2^j}  \sum_{k'=1'}^{2^{j'}} \int_{I^j_k \times I^{j'}_{k'}} f(y,y') (\psi^{j}_k(y) \times \psi^{j'}_{k'}(y')) dy dy' \chi_{I^j_k \times I^{j'}_{k'}}(x,x')
    \label{eq:14}
    \end{align}
    \end{itemize}

where $\chi_{I^j_k \times I^{j'}_{k'}}(x,x')$ is the characteristic function on $I^j_k \times I^{j'}_{k'}$ and $ \lVert \phi^j_k(x) \times \phi^{j'}_{k'}(x') \rVert_{L^2([0,1]^2)} = \lVert \phi^j_k(x) \times \psi^{j'}_{k'}(x') \rVert_{L^2([0,1]^2)} = \lVert \psi^j_k(x) \times \phi^{j'}_{k'}(x') \rVert_{L^2([0,1]^2)} = \lVert \psi^j_k(x) \times \psi^{j'}_{k'}(x') \rVert_{L^2([0,1]^2)} = 1$.
\label{defn:3.6}
\end{definition}

\begin{remark}
As discussed in \cite{mallat1989theory}, the relation between the wavelet and convolution operators is $Q^j = P^{j+1} - P^j$. This relation will be useful for compressing the approximation and providing estimates on the residual.
\label{rem:3.7}
\end{remark}

We will also construct a continuous - in the sense of the location and scaling parameters of the wavelet family - tensor paraproduct decomposition; to do so, we analogously define the continuous orthonormal wavelet and scaling functions, $\psi^t_u(x), \phi^t_u(x)$, where $(t,u) \in \mathbb{R}_{+} \times [0,1] $ and such that $t$ and $j$ are related by $t = 2^{-j}$. 

\begin{definition}
Continuous wavelet orthonormal and scaling functions, $(\psi^t_u(x),\phi^t_u(x))$:
\begin{align}
\psi^t_u(x) = \frac{1}{\sqrt{t}}\psi(\frac{x - u}{t}), \phi^t_u(x) = \frac{1}{\sqrt{t}}\phi(\frac{x - u}{t})
\label{eq:15}
\end{align}
\label{defn:3.6}
\end{definition}

$\psi^{t'}_{u'}(x'), \phi^{t'}_{u'}(x')$ are defined along the other direction of $[0,1]^2$ as done in the multiscale setting. Let the continuous tensor scaling operator, $P^tP'^{t'}(f)$ be defined as:

\begin{definition} Continuous tensor scaling operator, $P^tP'^{t'}(f)$:

\begin{align}
& P^t P'^{t'}(f)(u_m,u'_m)\coloneq \int_0^1 \int_0^1 f(y,y') (\phi_{u_m}^t(y) \times \phi_{u'_m}^{t'}(y')) dy dy'  \chi_{[0,1]^2}(u_m,u'_m) \nonumber \\
& P^t P'^{t'}(f) \coloneq \sum_{m=1}^{\infty} P^t P'^{t'}(f)(u_m,u'_m) , \hspace{0.2 cm} (u_m,u'_m)_{m=1}^{\infty} \in [0,1]^2
\label{eq:16}
\end{align}

where $\chi_{[0,1]^2}$ is the characteristic function on $[0,1]^2$ and 
$\lVert \phi^t_u(x) \times \phi^{t'}_{u'}(x') \rVert_{L^2([0,1]^2)} = 1$
\label{defn:3.7}
\end{definition}

The continuous and multiscale orthonormal wavelet families are key ingredients in constructing both the continuous and multiscale tensor paraproduct decompositions. For the multiscale construction, we need provide a few lemmas establishing a connection between Holder regularity and the expansion coefficients of the wavelet and tensor wavelet families they are processed in. Let $c^j_k$ and $d^j_k$ be the scaling and wavelet coefficients of $f$ on the interval $I^j_k$, respectively. 

\begin{align}
c^j_k \coloneq \int_{I^j_k} f(y) \phi^j_k(y) dy, d^j_k \coloneq \int_{I^j_k} f(y) \psi^j_k(y) dy  
\label{eq:17}
\end{align}

Then the following lemma from \cite{gavish2010multiscale} establishes a connection between the decay of the wavelet family coefficients and the holder regularity of $f$:

\begin{lemma}
Let $d^j_k $ and $c^j_k$ be the wavelet and scaling expansion coefficients at the interval $I^j_k$.  $|d^j_k| , |c^j_k| \leq 2^{-j(\alpha + \frac{1}{2})} \hspace{0.2 cm} \forall \hspace{0.2 cm} I^j_k \in [0,1]$ \text{iff} $f(x) \in \Lambda_{\alpha}([0,1])$
\label{lem:3.9}
\end{lemma}

\begin{proof}
See \cite{gavish2010multiscale}
\end{proof}

A similar result was proved in \cite{gavish2012sampling} for mixed Holder functions. Suppose $\alpha^{j,j'}_{k,k'}, \omega^{j,j'}_{k,k'}, \eta^{j,j'}_{k,k'}$ are the coefficients of $f \in \Lambda_{\alpha}([0,1]^2)$ in the tensor wavelet family supported on the rectangle $I^j_k \times I^{j'}_{k'}$. Then the following lemma holds:

\begin{lemma}
Let $\alpha^{j,j'}_{k,k'}, \omega^{j,j'}_{k,k'}, \eta^{j,j'}_{k,k'}$ be the expansion coefficients of $f(x,x')$ in the orthonormal tensor wavelet, tensor scaling-wavelet/wavelet-scaling, and tensor scaling bases, respectively. $|\alpha^{j,j'}_{k,k'}|, |\omega^{j,j'}_{k,k'}|, |\eta^{j,j'}_{k,k'}| \leq 2^{-(j+j')(\alpha + \frac{1}{2})} \hspace{0.2 cm} \forall \hspace{0.2 cm} I^j_k \times I^{j'}_{k'} \in [0,1]^2$ \text{iff} $f(x,x') \in \Lambda_{\alpha}([0,1]^2)$
\label{lem:3.10}
\end{lemma}

\begin{proof}
See \cite{gavish2012sampling}
\end{proof}

This result brings us to our lemma, which can be seen as a slight generalization of the previous one, as the assumption on the coefficients is less restrictive. It will be used to establish $2\alpha$-holder regularity of the residual $\Delta_{(N,N')}(A,f)$ (to be precisely characterized later) between $A(f)$ and its approximation $\tilde{A}_{(N,N')}(f)$. 

\begin{lemma}
Suppose $f(x,x')$ can be written as a sum over characteristic functions multiplied by constant coefficients supported on dyadic rectangles $I^j_k \times I^{j'}_{k'}$ with the estimate of $ | \beta^{j,j'}_{k,k'} | \leq 2^{-(j+j')(2\alpha + \epsilon)} , \epsilon > 0.\hspace{0.2 cm} \forall \hspace{0.2 cm} I^j_k \times I^{j'}_{k'} \in [0,1]^2$.  Then $f(x,x') \in \Lambda_{2\alpha}([0,1]^2)$
\label{lem:3.11}
\end{lemma}

\begin{proof}

Conceptually, the proof boils down to two cases. Pick 4 vertices of an arbitrary dyadic rectangle in $[0,1]^2$ and define a dyadic rectangle which is the infimum of all dyadic rectangles in $[0,1]^2$ containing the four points. In one case consider all the scales where the points are contained in the dyadic rectangle and surmise that the mixed difference is 0. In the other case, recall that all coefficients at a fixed scale have the same estimate to construct a geometric series which is bounded by the product of the dyadic distances along each axis by construction. The same reasoning holds for showing dyadic holder regularity along the individual axes which completes the proof. Now we proceed formally:

Let $\vec{x}_1 = (x,x')$, $\vec{x}_2 = (y,x')$, $\vec{x}_3 = (x,y')$, $\vec{x}_4 = (y,y') \in [0,1]^2$ be an arbitrary rectangle in $[0,1]^2$. To show

\begin{align}
\frac{|f(\vec{x}_1) - f(\vec{x}_2) -f(\vec{x}_3) + f(\vec{x}_4)|}{d_d(x,y)^{2\alpha}d'_d(x',y')^{2\alpha}} \leq C
\label{eq:18}
\end{align}

\begin{align}
\frac{|f(\vec{x}_1) - f(\vec{x}_2)|}{d_d(x,y)^{2\alpha}} \leq C
\label{eq:19}
\end{align}

\begin{align}
\frac{|f(\vec{x}_1) - f(\vec{x}_3)|}{d'_d(x',y')^{2\alpha}} \leq C
\label{eq:20}
\end{align}

define $\tilde{R}$ to be the smallest dyadic rectangle containing all four points: 

\begin{align}
\tilde{R} \coloneq \inf_{\substack{ (I^j_k \times I^{j'}_{k'}) \in \\ [0,1]^2}} \{ I^j_k \times I^{j'}_{k'} : \vec{x}_1, \vec{x}_2, \vec{x}_3, \vec{x}_4 \in I^j_k \times I^{j'}_{k'}\}
\label{eq:21}
\end{align}

Now construct a set comprised of a sequence of tuples of scaling parameters in each direction such that all associated dyadic rectangles are larger than $\tilde{R}$

\begin{align}
U = \{ (j,j') : | \tilde{R} | < 2^{-(j + j')} \}
\label{eq:22}
\end{align}

Similarly, construct a set comprised of a sequence of tuples of scaling parameters in each direction such that all associated dyadic rectangles are less than or equal to $\tilde{R}$

\begin{align}
L = \{ (j,j') : | \tilde{R} | \geq 2^{-(j + j')} \}
\label{eq:23}
\end{align}

For each $\vec{u} \in U$ let $\vec{u}(0)$ be the scaling parameter in direction $j$ and let $\vec{u}(1)$ be the scaling parameter in direction $j'$. Then one has

\begin{align}
& \sum_{\vec{u} \in U} \sum_{k=1}^{2^{\vec{u}(0)}} \sum_{k'=1}^{2^{\vec{u}(1)}} f(\vec{x}_1) - f(\vec{x}_2) - f(\vec{x}_3) + f(\vec{x}_4) = 0 \nonumber \\
& \leq d_d(x,y) d'_d(x',y')
\label{eq:24}
\end{align}

In the other case, also define $\vec{l}(0)$ be the scaling parameter in direction $j$ and let $\vec{l}(1)$ be the scaling parameter in direction $j'$. Then one has

\begin{align}
& \sum_{\vec{l} \in L} \sum_{k=1}^{2^{\vec{l}(0)}} \sum_{k'=1}^{2^{\vec{l}(1)}} f(\vec{x}_1) - f(\vec{x}_2) - f(\vec{x}_3) + f(\vec{x}_4) \nonumber \\
& \leq \sum_{\vec{l} \in L} C_{\vec{l}} 2^{-(\vec{l}(0) + \vec{l}(1))(2\alpha + \epsilon))} \nonumber \\
& = C 2^{-\vec{l}_m(2\alpha + \epsilon)} \nonumber \\
&  \leq d_d(x,y)^{2\alpha}d'_d(x',y')^{2\alpha}
\label{eq:25}
\end{align}

where $\vec{l}_m \coloneq \text{min} \{ (\vec{l}(0) + \vec{l}(1)) : \vec{l} \in L \}$. The first step comes from the assumption on the coefficients of $f(x,x')$, the second step from collapsing the geometric series, and the last step by construction of $L$ since $\tilde{R} = d_d(x,y)d'_d(x',y')$.

\end{proof}

\begin{lemma}
Let $\beta^{j,j'}_{k,k'}$ be an expansion coefficient of $f(x,x')$ in an orthonormal tensor  basis supported in the rectangle $I^j_k \times I^{j'}_{k'}$. If $f(x,x') \in \Lambda_{2 \alpha}([0,1]^2)$ then $|\beta^{j,j'}_{k,k'}| \leq 2^{-(j+j')(2\alpha + \frac{1}{2})} \hspace{0.2 cm} \forall \hspace{0.2 cm} I^j_k \times I^{j'}_{k'} \in [0,1]^2$.
\label{lem:3.12}
\end{lemma}

\begin{proof}
Suppose $f \in \Lambda_{2 \alpha} ([0,1]^2)$. Let $(x,x'), (y,x'), (x,y'), (y,y') \in I^j_k \times I^{j'}_{k'} \subset [0,1]^2$ be the vertices of an arbitrary rectangle. Then the following inequality holds:

\begin{align}
& | \beta^{j,j'}_{k,k'} |^2 = |<f(x,x'), \psi^j_k(x) \times \psi^{j'}_{k'}(x')>|^2 \nonumber \\
& = ( \int_{I^j_k \times I^{j'}_{k'}} f(x,x') \psi^j_k(x) \times \psi^{j'}_{k'}(x') dx dx' )^2 \nonumber \\
& = (\int_{I^j_k \times I^{j'}_{k'}} (f(x,x') - f(y,x')  - f(x,y') + f(y,y')) \psi^j_k(x) \times \psi^{j'}_{k'}(x') dx dx')^2 \nonumber \\
& \leq \lVert (f(x,x') - f(y,x')  - f(x,y') + f(y,y')) \rVert_{L^2(I^j_k \times I^{j'}_{k'})}^2  \lVert  \psi^j_k(x) \times \psi^{j'}_{k'}(x') \rVert_{L^2(I^j_k \times I^{j'}_{k'})}^2 \nonumber \\
& = (\int_{I^j_k \times I^{j'}_{k'}} (f(x,x') - f(y,x')  - f(x,y') + f(y,y'))^2 dx dx') \nonumber \\
& \leq \int_{I^j_k \times I^{j'}_{k'}} C_{(j,j')}^2 (d_d(x,y)^{2\alpha} d'_d(x',y')^{2\alpha})^2 dx dx' \nonumber \\
& = C_{(j,j')}^2 d_d(x,y)^{4\alpha} d'_d(x,y)^{4\alpha} d_d(x,y)^{\alpha}d'_d(x,y)^{\alpha} \nonumber \\
& =  C_{(j,j')}^2 2^{-(j+j')(4\alpha + 1)}  \nonumber \\
\label{eq:26}
\end{align}

Thus, $| \beta^{j,j'}_{k,k'} | \leq 2^{-(j+j')(2\alpha + \frac{1}{2})}$ since $\lVert  \psi^j_k(x) \times \psi^{j'}_{k'}(x')  \rVert_{L^2(I^j_k \times I^{j'}_{k'})} = 1$ and the selection of the dyadic rectangle is arbitrary, the result holds for all dyadic rectangles, thus completing the proof.

\end{proof}

If we strengthen the assumptions in Lemma \ref{lem:3.11} by requiring $\epsilon \geq \frac{1}{2}$ we obtain the following bidirectional result.

\begin{lemma}
Suppose $\beta^{j,j'}_{k,k'} $ is an expansion coefficient of $f(x,x')$ in an orthonormal tensor basis supported on the dyadic rectangle $I^j_k \times I^{j'}_{k'}$. One obtains $ |\beta^{j,j'}_{k,k'}| \leq 2^{-(j+j')(2\alpha + \epsilon)}, \epsilon \geq \frac{1}{2} \hspace{0.2 cm} \forall \hspace{0.2 cm} I^j_k \times I^{j'}_{k'} \in [0,1]^2$ \text{iff} $f(x,x') \in \Lambda_{2\alpha}([0,1]^2)$
\label{lem:3.13}
\end{lemma}

\begin{proof}

Set $\epsilon \geq \frac{1}{2}$ to observe the bidirectional implication of Lemmas \ref{lem:3.11} and \ref{lem:3.12}.

\end{proof}

As we use the tensor wavelet, tensor scaling-wavelet, tensor wavelet-scaling, and tensor scaling convolution operators to construct the tensor paraproduct decomposition, we provide the following corollary to contextualize Lemma \ref{lem:3.13}.

\begin{corollary}
Suppose $\alpha^{j,j'}_{k,k'}, \omega^{j,j'}_{k,k'}, \eta^{j,j'}_{k,k'}$ are expansion coefficients of $f(x,x')$ in the orthonormal tensor wavelet, tensor scaling-wavelet, and tensor scaling bases each supported on the dyadic rectangle $I^j_k \times I^{j'}_{k'}$. One obtains $|\alpha^{j,j'}_{k,k'}|, |\omega^{j,j'}_{k,k'}|, |\eta^{j,j'}_{k,k'}| \leq 2^{-(j+j')(2\alpha + \epsilon)}, \epsilon \geq \frac{1}{2} \hspace{0.2 cm} \forall \hspace{0.2 cm} I^j_k \times I^{j'}_{k'} \in [0,1]^2$ \text{iff} $f(x,x') \in \Lambda_{2\alpha}([0,1]^2)$
\label{cor:3.14}
\end{corollary}

\begin{proof}

Special case of Lemma \ref{lem:3.13}

\end{proof}

\section{Tensor Paraproducts}

\begin{comment}
\begin{enumerate}
    \item 1D Bony Lemma (no proof)
    \item 2D continuous
    \item 2D discrete
    \item Commentary on $A'$, $A''$ as symbols/explicit construction/statement of paralinearization operator
    \item Commentary on operator interpolation
\end{enumerate}
\end{comment}

Here we explicitly construct our continuous and multiscale tensor paraproduct decompositions. We use the modifier \emph{tensor} to emphasize that we use distinct scaling parameters for each axis of $[0,1]^2$ to construct the paraproduct, as opposed to a single scaling parameter. We begin by stating the 1D paraproduct result by Bony.

\begin{lemma}
Suppose $A \in C^{\infty}(\mathbb{R})$ and $f \in \Lambda_{\alpha}(\mathbb{R}^d)$. Bony showed in \cite{bony1981calcul} $T : f \to A(f)$ can be written as:

\begin{align}
\Pi_{A'(P^{(j < C)}(f))}: f  \to A'(P^{(j < C)}(f))(P^{(j \geq C)}(f)) + R(A,f) 
\label{eq:27}
\end{align}

where $P^{(j < C)}(f)$ and $P^{(j \geq C)}(f)$ select the low and high frequencies of $f$, respectively which are separated by some threshold frequency $C$ and $R(A,f)$ is a smooth residual.

\label{lem:4.1}
\end{lemma}

Lemma \ref{lem:4.1} had remarkable implications for problems in nonlinear PDE, which we illustrate with the following example.

\begin{example}
Consider the problem $\frac{d}{dt} f(t) = e^{f(t)}$. Assume $f(t) \in \Lambda_{\alpha}(\mathbb{R}^d)$. Then the nonlinear action of $e$ on $f(t)$ is replaced by $e^{(P^{(j < C)}(f))}P^{(j \geq C)}(f)$ which is a constant multiplied by the high frequencies of $f$, transforming the ODE $\frac{d}{dt} f(t) = e^{f(t)}$ into a simpler form: $\frac{d}{dt} f(t) = e^{(P^{(j < C)}(f))}P^{(j \geq C)}(f) = C_fP^{(j \geq C)}(f) $
\label{ex:4.2}
\end{example}

We now state the continuous version of our tensor paraproduct decomposition with the following lemma:

\begin{lemma}

Suppose  $A \in \mathcal{C}^2(\mathbb{R})$, $A(0)=0$, and $ f \in \Lambda_\alpha([0,1]^2), 0 < \alpha < \frac{1}{2}$, then $A(f)$ can be rewritten as:

\begin{align}
& A_{(t,t')}(f) =  \int_0^{\infty}\int_0^{\infty}  t' t \left[ A'(P^tP'^{t'}(f)) [\frac{\partial }{\partial t'} \frac{\partial }{\partial t} (P^{t}P'^{t'}(f))] \right.  \nonumber \\
& + \left. A''(P^tP'^{t'}(f)) [\frac{\partial }{\partial t'} (P^tP'^{t'}(f))] [\frac{\partial }{\partial t}(P^{t}P'^{t'}(f))] \right]  \frac{dt}{t}  \frac{dt'}{t'} \nonumber \\
& + A(P^0P'^{\infty}(f)) + A(P^{\infty}P'^{0}(f)) - A(P^{\infty}P'^{\infty}(f))
\label{eq:28}
\end{align}

replacing the operator $T : f \to A(f)$ with the continous paraproduct $\Pi^{(t,t')}_{(A',A'')} : f \to A_{(t,t')}(f)$

\end{lemma}

\begin{proof}
The proof follows by writing the action of $A$ on the mixed difference between continuous tensor scaling  convolution operators, observing $A(f) = \lim_{t \to 0} A(P^tP'^{t'}(f))$, and the fundamental theorem of calculus with respect to the scaling parameters $(t,t')$.
\end{proof}

\begin{remark}
The functions $A(P^0P'^{\infty}(f)) + A(P^{\infty}P'^{0'}(f)) - A(P^{\infty}P'^{\infty'}(f))$ evaluate to 0 since convolution with the scaling function at $t = \infty$ yields 0
\label{rem:4.4}
\end{remark}

\begin{remark}
If $A$ is nonlinear, Lemma 4.3 quasilinearizes $A$ by replacing $A$ with its derivative evaluated at the averages of $f$ i.e. $A'(P^tP'^{t'}(f))$ and $A''(P^tP'^{t'}(f))$. Furthermore, $\frac{\partial }{\partial t'} \frac{\partial }{\partial t} (P^{t}P'^{t'}(f))$ can be interpreted as the tensor wavelet coefficients of $f$ and $\frac{\partial }{\partial t'} (P^tP'^{t'}(f))\frac{\partial }{\partial t}(P^{t}P'^{t'}(f))$ as the product of the wavelet coefficients of $f$ in the $t$ and $t'$ directions. 
\label{rem:4.5}
\end{remark}

\begin{remark}
Since differentiation of $A$ extracts its high frequency content and the continuous tensor scaling operators extract the low frequency content of $f$, the terms $A'(P^tP'^{t'}(f))$ and $A''(P^tP'^{t'}(f))$ represent the high frequency content of $A$ acting on the low frequency content of $f$ and represent a high-low type paraproduct, while the wavelet coefficients, $\frac{\partial }{\partial t'} \frac{\partial }{\partial t} (P^{t}P'^{t'}(f)), \frac{\partial }{\partial t'} (P^tP'^{t'}(f))\frac{\partial }{\partial t}(P^{t}P'^{t'}(f))$, extract the high frequency content of $f$. 
\label{rem:4.6}
\end{remark}

\begin{remark}
Lemmas 4.1 and 4.3 can be used to delineate between paraproducts and what we call tensor paraproducts. In Lemma 4.1, even though the support of $f$ is $d$-dimensional, one scaling parameter is used to seperate the low and high frequency content. However in Lemma 4.3, we use two scales for each axis of the $2$-dimensional support of $f$, allowing for finer analytical control on each axis.
\label{rem:4.7}
\end{remark}

The multiscale tensor paraproduct decomposition is more involved as the continuous scaling parameters $(t,t')$ need to be discretized, requiring one to make an approximation. To achieve this, we construct a bilinear interpolation between the scaling parameters of the tensor scaling convolution operators such that the fundamental theorem of calculus can be applied, then make an approximation to circumvent integration with resepect to the scaling parameters. More work is required, but it permits one to implement the multiscale decomposition computationally.

\begin{theorem}
Suppose  $A \in \mathcal{C}^2(\mathbb{R})$, $ f \in  \Lambda_\alpha([0,1]^2), 0 < \alpha < \frac{1}{2}$, then for the operator $T: f \to A(f)$ we can approximate $A(f)$ with

\begin{align}
&  \tilde{A}_{(N,N')}(f) =  \sum_{j=0}^{N} \sum_{j'=0}^{N'}  A'(P^jP'^{j'}(f)) Q^jQ'^{j'}(f) + A''(P^jP'^{j'}(f))Q^jP'^{j'}(f)P^jQ'^{j'}(f)  
\label{eq:29}
\end{align}

such that the multiscale tensor paraproduct transforms $T : f \to A(f)$ to 

\begin{align}
\Pi^{(N,N')}_{(A',A'')} : f \to \tilde{A}_{(N,N')}(f) + \Delta_{(N,N')}(A,f)
\label{eq:30}
\end{align}

where $\Delta_{(N,N')}(A,f) = A(f) - \tilde{A}_{(N,N')}(f) \in \Lambda_{2\alpha}([0,1]^2)$ is the residual which has twice the regularity of $f$ and 

\begin{align}
\lVert \Delta_{(N,N')}(A,f) \rVert_{\Lambda_{2\alpha}([0,1]^2)} \leq C_A \lVert f \rVert_{\Lambda_{\alpha}([0,1]^2)}
\label{eq:31}
\end{align}

\label{thm:4.8}
\end{theorem}

\begin{proof}

First we construct the approximation $ \tilde{A}_{(N,N')}(f)$ to build $\Pi^{(N,N')}_{(A',A'')}$ while characterization of the residual comes in section 5. We buiild the approximation in three steps: (1) express $A(f)$ as a telescoping series comprised of the mixed difference between the scaling parameters of the tensor scaling convolution operators (2) construct a bilinear interpolation between the scaling parameters of the tensor scaling convolution operators, permitting one to apply the fundamental theorem of calculus (3) build the approximation by circumventing integration with respect to the scaling parameters. We proceed with step (1):

\begin{align}
& \sum_{j=0}^{N} \sum_{j'=0}^{N'}  A(P^{j+1}P'^{j'+1}(f)) - A(P^{j}P^{j'+1}(f)) - A(P^{j+1}P'^{j'}(f)) + A(P^{j}P'^{j'}(f)) \nonumber \\
&  = A(P^{N+1}P'^{N'+1}(f)) - A(P^0 P'^{N'+1}(f)) - A(P^{N+1}P'^{0'}(f)) + A(P^0P'^{0'}(f)) \nonumber \\
& = A(f) - C_{0,N'+1} - C_{N+1,0} + C_{0,0} + O(2^{-((N+2) + (N'+2))}) \nonumber \\
& = A_{(N,N')}(f)
\label{eq:32}
\end{align}

The first step comes from collapsing the telescoping series and the second step comes from noting $P^{j}P'^{j'}(f) \to f$ as $j \to \infty, j' \to  \infty $. Note that under the assumption $ \int_0^1 f(y,y') dy = 0$ or $\int_0^1 f(y,y') dy' = 0$, the constants vanish and $A(f)$ is equivalent to the first line modulo a constant on the order of the smallest dyadic rectangle. Next, we define a bilinear interpolation between the scaling parameters of the tensor scaling convolution operators such that evaluation of the interpolation parameters at the boundaries recovers the mixed difference:

\begin{definition}
We call $h_{\mu,\omega}(j,j',f)$ the bilinear tensor scaling interpolation operator ($\mu,\omega \in [0,1]$), and define its action on $f$ by:

\begin{align}
& h_{\mu,\omega}(j,j',f) \coloneq  h(\omega, \mu, P^{j+1}P'^{j'+1}(f), P^{j}P'^{j'+1}(f), P^{j+1}P'^{j'}(f), P^{j}P'^{j'}(f)) \nonumber \\
& \coloneq  \omega( P^{j+1}P'^{j'}(f) - P^{j}P'^{j'}(f) ) + \mu  ((P^{j}P'^{j'+1}(f) +  \nonumber \\
& \omega(P^{j+1}P'^{j'+1}(f) - P^{j}P'^{j'+1}(f))) - \nonumber \\ 
& (P^{j}P'^{j'}(f) + \omega(P^{j+1}P'^{j'}(f) - P^{j}P'^{j'}(f))))
\label{eq:33}
\end{align}

\label{def:4.9}
\end{definition}

\begin{remark}
 $h_{\mu,\omega}(j,j',f)$ is necessary to construct the multiscale tensor paraproduct, since it allows one to connect the mixed difference between the scales of the tensor scaling convolution operators to the fundamental theorem of calculus. It can be conceptually understood as a bilinear interpolation between the vertices of a rectangle, where the vertices are the scaling parameters. The interpolation parameter $\omega$ is the interpolates between scales in direction $j$ while the interpolation parameter $\mu$ interpolates between scales in direction $j'$ of the tensor scaling convolution operators. 
 \label{rem:4:10}
\end{remark}

Let's proceed writing $A_{(N,N')}(f)$ in terms of the mixed difference yields:

\begin{align}
& A_{(N,N')}(f) = \sum_{j=0}^{N} \sum_{j'=0}^{N'}  A(P^{j+1}P'^{j'+1}(f)) - A(P^{j}P'^{j'+1}(f)) - A(P^{j+1}P'^{j'}(f)) + A(P^{j}P'^{j'}(f)) \nonumber \\
&  = \sum_{j=0}^{N} \sum_{j'=0}^{N'} \int_0^1 \int_0^1 \frac{\partial}{\partial \mu} \frac{\partial}{\partial \omega} A(P^jP'^{j'}(f) + h_{\mu,\omega}(j,j',f)) d\mu d\omega \nonumber \\
& = \sum_{j=0}^{N} \sum_{j'=0}^{N'} \int_0^1 \int_0^1 \frac{\partial}{\partial \mu} (A'(P^jP'^{j'}(f) + h_{\mu,\omega}(j,j',f))  \frac{\partial}{\partial \omega} (P^jP'^{j'}(f) + h_{\mu,\omega}(j,j',f)))  d\mu d\omega \nonumber \\
& = \sum_{j=0}^{N} \sum_{j'=0}^{N'} \int_0^1 \int_0^1 A'(P^jP'^{j'}(f) + h_{\mu,\omega}(j,j',f))\frac{\partial}{\partial \mu} \frac{\partial}{\partial \omega} (P^jP'^{j'}(f) + h_{\mu,\omega}(j,j',f)) + \nonumber \\ 
& A''(P^jP'^{j'}(f) + h_{\mu,\omega}(j,j',f)) \frac{\partial}{\partial \mu} (P^jP'^{j'}(f) + h_{\mu,\omega}(j,j',f)) \frac{\partial}{\partial \omega} (P^jP'^{j'}(f) + h_{\mu,\omega}(j,j',f))d\mu d\omega \nonumber \\
\label{eq:34}
\end{align}

where the first step is an identity, the second step comes from applying the fundamental theorem of calculus, and the subsequent steps come from application of the partial derivatives. Computing the partial derivatives, one obtains:

\begin{align}
&\frac{\partial}{\partial \mu} (P^jP'^{j'}(f) + h_{\mu,\omega}(j,j',f)) = (P^jP'^{j' + 1}(f) + \omega(P^{j+1}P'^{j' + 1}(f) - P^jP'^{j' + 1}(f))) \nonumber \\
& - (P^jP'^{j'}(f) + \omega(P^{j+1}P'^{j'}(f) - P^jP'^{j'}(f)))
\label{eq:35}
\end{align}

\begin{align}
&\frac{\partial}{\partial \omega} (P^jP'^{j'}(f) + h_{\mu,\omega}(j,j',f)) = \nonumber \\
& (P^{j+1}P'^{j'}(f) - P^jP'^{j'}(f)) + \mu ((P^{j+1}P'^{j'+1}(f) - P^jP'^{j'+ 1}(f)) - (P^{j+1}P'^{j'}(f) - P^jP'^{j'}(f)) )
\label{eq:36}
\end{align}

\begin{align}
& \frac{\partial}{\partial \mu}\frac{\partial}{\partial \omega} (P^jP'^{j'}(f) + h_{\mu,\omega}(j,j',f)) =  P^{j+1}P'^{j'+1}(f) - P^{j}P'^{j'+1}(f) - P^{j+1}P'^{j'}(f) + P^{j}P'^{j'}(f) 
\label{eq:37}
\end{align}

\begin{comment}
& = \sum_{j=0}^{N} \sum_{j'=0}^{N'}\int_0^1 \int_0^1 A'(P^{j}P'^{j'}(f) +  h(\omega, \mu, P^{j+1}P'^{j'+1}(f), P^{j}P'^{j'+1}(f), P^{j+1}P'^{j'}(f), P^{j}P'^{j'}(f)) \nonumber \\
& ((P^{j+1}P'^{j'+1}(f) - P^{j}P'^{j'+1}(f) - P^{j+1}P'^{j'}(f)) + P^{j}P'^{j'}(f)) + \nonumber \\
& A''(P^{j}P'^{j'}(f) +  h(\omega, \mu, P^{j+1}P'^{j'+1}(f), P^{j} P'^{j'+1}(f), P^{j+1}P'^{j'}(f), P^{j}P'^{j'}(f))) \nonumber \\
& ((P^{j+1}P'^{j'}(f)  - P^{j}P'^{j'}(f)) + \mu ((P^{j+1}P'^{j'+1} - P^jP'^{j'+1}) - (P^{j+1}P'^{j'} - P^j P'^{j'}))  \nonumber \\
& ((P^{j}P'^{j'+1}(f)  + \omega(P^{j+1}P'^{j'+1}(f) - P^{j}P'^{j'+1}(f))) \nonumber \\
& -(P^{j}P'^{j'}(f) + \omega(P^{j+1}P'^{j'}(f) - P^{j}P'^{j'}(f))))   d \mu d \omega \nonumber \\
\end{comment}

To circumvent integration with respect to $\mu,\omega$, we build an approximation which omits terms with dependence on those integration variables. First we define the multiplicative factors, $\mathbf{v}_1^{(j,j')}(f)$ and $\mathbf{v}_2^{(j,j')}(f)$ for the $A'$  and $A''$ terms, respectively:

\begin{align}
\mathbf{v}_1^{(j,j')}(f) \coloneq P^{j+1}P'^{j'+1}(f) - P^{j}P'^{j'+1}(f) - P^{j+1}P'^{j'}(f)) + P^{j}P'^{j'}(f) 
\label{eq:38}
\end{align} 

\begin{align}
& \mathbf{v}_2^{(j,j')}(f) \coloneq  (P^jP'^{j' + 1}(f) + \omega(P^{j+1}P'^{j' + 1}(f) - P^jP'^{j' + 1}(f))) \nonumber \\
& - (P^jP'^{j'}(f) + \omega(P^{j+1}P'^{j'}(f) - P^jP'^{j'}(f)))(P^{j+1}P'^{j'}(f) - P^jP'^{j'}(f)) + \nonumber \\
& \mu ((P^{j+1}P'^{j'+1}(f) - P^jP'^{j'+ 1}(f)) - (P^{j+1}P'^{j'}(f) - P^jP'^{j'}(f)) )
\label{eq:39}
\end{align}

Observe that $\mathbf{v}^{(j,j')}_1(f)$ is the mixed difference between scales of the tensor scaling convolution operators and $\mathbf{v}^{(j,j')}_2(f)$ is the product of the partial derivatives with respect to each interpolation parameter. This allows us to compactly express $A_{(N,N')}(f)$ as:

\begin{align}
& A_{(N,N')}(f) = \sum_{j=0}^{N} \sum_{j'=0}^{N'} \int_0^1 \int_0^1 A'(P^jP'^{j'}(f) + h_{\mu,\omega}(j,j',f))\mathbf{v}_1^{(j,j')}(f) + \nonumber \\ 
& A''(P^jP'^{j'}(f) + h_{\mu,\omega}(j,j',f))\mathbf{v}_2^{(j,j')}(f) d\mu d\omega \nonumber \\
\label{eq:40}
\end{align}

To circumvent integration with respect to the interpolation parameters $\mu,\omega$, we define $\tilde{\mathbf{v}}^{(j,j')}_2(f)$ to be an approximation to $\mathbf{v}^{(j,j')}_2(f)$.

\begin{align}
\mathbf{\tilde{v}}^{(j,j')}_2(f) \coloneq  (P^{j}P'^{j'+1}(f) - P^{j}P'^{j'}(f))(P^{j+1}P'^{j'}(f) - P^{j}P'^{j'}(f))  
\label{eq:41}
\end{align}

The approximation is made by only keeping the terms in each product of $\mathbf{v}^{(j,j')}_2(f)$ that don't have explicit dependence on the interpolation parameters $\mu,\omega$, yielding:

\begin{align}
& \tilde{A}_{(N,N')}(f) = \sum_{j=0}^{N} \sum_{j'=0}^{N'}  A'(P^jP'^{j'}(f))\mathbf{v}_1^{(j,j')}(f) + A''(P^jP'^{j'}(f))\tilde{\mathbf{v}}^{(j,j')}_2(f) \nonumber \\
\label{eq:42}
\end{align}

Recalling $Q^j = P^{j+1} - P^j$ from Remark \ref{rem:3.7}, we rewrite $\tilde{A}_{(N,N')}(f)$ as:

\begin{align}
& \tilde{A}_{(N,N')}(f) = \sum_{j=0}^{N} \sum_{j'=0}^{N'}  A'(P^jP'^{j'}(f))Q^jQ'^{j'}(f) +  A''(P^jP'^{j'}(f))P^jQ'^{j'}(f)Q^jP'^{j'}(f) 
\label{eq:43}
\end{align}

and the residual between $A_{(N,N')}(f)$ and its approximation $\tilde{A}_{(N,N')}(f)$:

\begin{align}
&\Delta_{N,N'}(A,f) = \sum_{j=0}^N \sum_{j'=0}^{N'} \int_0^1 \int_0^1 (A'(P^jP'^{j'}(f) + h_{\mu,\omega}(j,j',f))\mathbf{v}_1^{(j,j')}(f) - A'(P^jP'^{j'}(f)) \mathbf{v}_1^{(j,j')}(f)) + \nonumber \\
& (A''(P^jP'^{j'}(f)+ h_{\mu,\omega}(j,j',f))\mathbf{v}_2^{(j,j')}(f) - A''(P^jP'^{j'}(f))\tilde{\mathbf{v}}_2^{(j,j')}(f)) d\mu d \omega
\label{eq:44}
\end{align}

Appealing to Theorem \ref{prop:5.2} in section 5 completes the proof. 

\end{proof}

\begin{remark}
Observe the multiscale approximation, $\tilde{A}_{(N,N')}(f)$, resembles the continuous representation, $\tilde{A}_{(t,t')}(f)$; and has the same frequency separation properties in the following sense. The principal term $A'(P^jP'^{j'})$ is a high-low tensor paraproduct since $A'$ extracts the high frequency content of $A$ while $P^jP'^{j'}$ extracts the low-frequency content of $f$. The same reasoning applies for the second order term $A''(P^jP'^{j'})$. $Q^jQ^{j'}$ is a high-high tensor paraproduct since it extracts the high frequency content in both directions. Furthermore, $Q^jP'^{j'}$ is a high-low tensor paraproduct as it extracts the high frequency content of $f$ in direction $j$ and the low frequency content of $f$ in direction $j'$ while $P^jQ'^{j'}$ is a low-high tensor paraproduct since the directions of frequency extraction are reversed. 
\label{rem:4.11}
\end{remark}

Now that the nature of $\tilde{A}_{(N,N')}(f)$ has been characterized, we proceed with qualitative and quantitative characterization of the residual in the following section.

\section{Residual Estimates}

\begin{comment}
\begin{enumerate}
    \item Quantitative
    \item Qualitative
\end{enumerate}
\end{comment}

We characterize the residual by providing quantitative and qualitative estimates by measuring the residual in $L^{\infty}$. First we compress the representations of the functions $\textbf{v}_1^{(j,j')}(f), \textbf{v}_2^{(j,j')}(f), \tilde{\textbf{v}}_2^{(j,j')}(f)$ defined in the residual, $\Delta_{(N,N')}(A,f)$, computed in (\ref{eq:44}).

\begin{align}
\textbf{v}_1^{(j,j')}(f) \coloneq Q^jQ'^{j'}(f)
\label{eq:45}
\end{align}

\begin{align}
\textbf{v}_2^{(j,j')}(f) \coloneq (P^jQ'^{j'}(f) + \omega Q^jQ^{j'}(f))(Q^jP'^{j'}(f) + \mu Q^jQ^{j'}(f))
\label{eq:46}
\end{align}

\begin{align}
\tilde{\textbf{v}}_2^{(j,j')}(f) \coloneq Q^jP'^{j'}(f)P^jQ'^{j'}(f)
\label{eq:47}
\end{align}

These representations are obtained by recalling Remark \ref{rem:3.7} Now we state the first proposition characterizing the residual:

\begin{proposition}
    
If we measure the residual in $L^{\infty}$, we obtain:
\begin{align}
\lVert \Delta_{(N,N')}(A,f) \lVert_{L^{\infty}([0,1]^2)} = \sum_{j=0}^N \sum_{j'=0}^{N'} C_{(j,j')} 2^{-(j+j')(2\alpha + 1)}
\label{eq:48}
\end{align}
such that we can rewrite the residual as a sum of constant coefficients supported on dyadic rectangles:

\begin{align}
\Delta_{(N,N')}(A,f) \coloneq \sum_{j=0}^N \sum_{j'=0}^{N'} \sum_{k=1}^{2^j} \sum_{k'=1}^{2^{j'}}  \beta^{j,j'}_{k,k'} \chi_{I^j_k \times I^{j'}_{k'}}, | \beta^{j,j'}_{k,k'} | \leq C_{(j,j')}2^{-(j+j')(2\alpha + 1)}
\label{eq:49}
\end{align}

where $\chi_{I^j_k \times I^{j'}_{k'}}$ is the characteristic function on a dyadic rectangle.

\label{prop:5.1}
\end{proposition}

\begin{proof}

First observe:

\begin{align}
& \lVert \Delta_{(N,N')}(A,f) \rVert_{L^{\infty}([0,1]^2)} = \lVert \sum_{j=0}^N \sum_{j'=0}^{N'} \int_0^1 \int_0^1 (A'(P^jP'^{j'}(f) + h_{\mu,\omega}(j,j',f))\mathbf{v}^{(j,j')}_1(f) - A'(P^jP'^{j'}(f)) \mathbf{{v}}_1^{(j,j')}(f)) + \nonumber \\
& A''(P^jP'^{j'}(f)+ h_{\mu,\omega}(j,j',f))\mathbf{v}_2^{(j,j')}(f) - A''(P^jP'^{j'}(f)) \mathbf{\tilde{v}}_2^{(j,j')}(f) d\mu d \omega \rVert_{L^{\infty}([0,1]^2)} \nonumber \\
& \leq \sum_{j=0}^N \sum_{j'=0}^{N'}  \lVert \int_0^1 \int_0^1 (A'(P^jP'^{j'}(f) + h_{\mu,\omega}(j,j',f))\mathbf{v}^{(j,j')}_1(f) - A'(P^jP'^{j'}(f)) \mathbf{{v}}_1^{(j,j')}(f)) + \nonumber \\
& A''(P^jP'^{j'}(f)+ h_{\mu,\omega}(j,j',f))\mathbf{v}_2^{(j,j')}(f) - A''(P^jP'^{j'}(f)) \mathbf{\tilde{v}}_2^{(j,j')}(f) d\mu d \omega \rVert_{L^{\infty}([0,1]^2)} \nonumber \\
& \leq \sum_{j=0}^N \sum_{j'=0}^{N'} \lVert \int_0^1 \int_0^1 A'(P^jP'^{j'}(f) + h_{\mu,\omega}(j,j',f))\mathbf{v}_1^{(j,j')}(f) d\mu d\omega \rVert_{L^{\infty}([0,1]^2)}  +  \nonumber \\
& \lVert \int_0^1 \int_0^1 A'(P^jP'^{j'}(f)) \mathbf{v}_1^{(j,j')}(f) d \mu d \omega \rVert_{L^{\infty}([0,1]^2)} + \nonumber \\
& \lVert \int_0^1 \int_0^1 A''(P^jP'^{j'}(f)+ h_{\mu,\omega}(j,j',f))\mathbf{v}^{(j,j')}_2(f) d \mu d \omega \rVert_{L^{\infty}([0,1]^2)} + \nonumber \\
&\lVert \int_0^1 \int_0^1 A''(P^jP'^{j'}(f)) \mathbf{\tilde{v}}^{(j,j')}_2(f) d \mu d\omega \rVert_{L^{\infty}([0,1]^2)}
\label{eq:50}
\end{align}

by triangle inequality. Estimating each of the terms in (\ref{eq:50}) seperately, we have:

\begin{align}
& \lVert \int_0^1 \int_0^1 A'(P^jP'^{j'}(f) + h_{\mu,\omega}(j,j',f))\mathbf{v}_1^{(j,j')}(f) d\mu d\omega \rVert_{L^{\infty}([0,1]^2)} \leq \nonumber \\
&  \lVert \sup_{\mu, \omega \in [0,1]} \int_0^1 \int_0^1  A'(P^jP'^{j'}(f) + h_{\mu,\omega}(j,j',f))\mathbf{v}_1^{(j,j')}(f) d\mu d\omega \rVert_{L^{\infty}([0,1]^2)} \nonumber \\ 
& = \lVert  A'(P^jP'^{j'}(f) + Q^jQ'^{j'}(f))Q^jQ'^{j'}(f) \rVert_{L^{\infty}([0,1]^2)} \nonumber \\ 
& \leq \sum_{k=1}^{2^j} \sum_{k'=1}^{2^{j'}} \lVert  A'(P_k^jP_{k'}^{j'}(f) + Q_k^jQ_{k'}'^{j'}(f))Q_k^jQ_{k'}'^{j'}(f) \rVert_{L^{\infty}(I^j_k \times I^{j'}_{k'})}
\label{eq:51}
\end{align}

where the last step holds by triangle inequality and we substitute $\mathbf{v}^{(j,j')}_1(f)$ with $Q^jQ'^{j'}(f)$. Now consider arbitrary $x,z \in I^j_k \times I^{j'}_{k'}$ for $k = 1, \ldots, 2^j, k' = 1, \ldots, 2^{j'}$ to obtain:

\begin{align}
& \sum_{k=1}^{2^j} \sum_{k'=1}^{2^{j'}} \lVert  A'(P_k^jP_{k'}'^{j'}(f) + Q_k^jQ_{k'}'^{j'}(f))Q_k^jQ_{k'}'^{j'}(f) \rVert_{L^{\infty}(I^j_k \times I^{j'}_{k'})} \nonumber \\ 
& = \sum_{k=1}^{2^j} \sum_{k'=1}^{2^{j'}} \lVert A(P_k^jP_{k'}'^{j'}(f)(x) + Q_k^jQ_{k'}'^{j'}(f)(x)) - A(P_k^jP_{k'}'^{j'}(f)(z) + Q_k^jQ_{k'}'^{j'}(f)(z)) Q_k^jQ_{k'}'^{j'}(f) \rVert_{L^{\infty}(I^j_k \times I^{j'}_{k'})} \nonumber \\
& \leq \sum_{k=1}^{2^j} \sum_{k'=1}^{2^{j'}}  \lVert C_{(j,j') }(P_k^jP_{k'}'^{j'}(f)(x) + Q_k^jQ_{k'}'^{j'}(f)(x)) - (P_k^jP_{k'}'^{j'}(f)(z) + Q_k^jQ_{k'}'^{j'}(f)(z)) Q_k^jQ_{k'}'^{j'}(f) \rVert_{L^{\infty}(I^j_k \times I^{j'}_{k'})}
\label{eq:52}
\end{align}

where the last two steps hold since $A \in \mathcal{C}^2(\mathbb{R})$ and the range of the codomain of $P^j_k P'^{j'}_{k'}(f)$ is finite since $f$ is supported on $I^j_k \times I^{j'}_{k'}$.  Let $\alpha^{(j,j')}_{(k,k')} \chi_{(I^j_k \times I^{j'}_{k'})} \coloneq Q_k^jQ_{k'}'^{j'}(f) , \eta^{(j,j')}_{(k,k')} \chi_{(I^j_k \times I^{j'}_{k'})} \coloneq P_k^jP_{k'}'^{j'}(f)$ be the the expansion coefficient supported on the characteristic rectangle for the mixed wavelet function convolution operator and mixed scaling function convolution operators, respectively. Rewriting equation (\ref{eq:52}) gives:

\begin{align}
& \sum_{k=1}^{2^j} \sum_{k'=1}^{2^{j'}} \lVert ((\eta^{(j,j')}_{(k,k')} + \alpha^{(j,j')}_{(k,k')})\chi_{(I^j_k \times I^{j'}_{k'})}(x) - (\eta^{(j,j')}_{(k,k')} + \alpha^{(j,j')}_{(k,k')})\chi_{(I^j_k \times I^{j'}_{k'})}(z)) \alpha^{(j,j')}_{(k,k')}\chi_{(I^j_k \times I^{j'}_{k'})} \rVert_{L^{\infty}(I^j_k \times I^{j'}_{k'})} \nonumber \\
& \leq \sum_{k=1}^{2^j} \sum_{k'=1}^{2^{j'}} C_{(j,j')} 2^{-(j+j')\alpha + \frac{1}{2}} 2^{-(j+j')\alpha + \frac{1}{2}} \nonumber \\
& = C_{(j,j')} 2^{-(j+j')(2\alpha + 1)}
\label{eq:53}
\end{align}

where the last step holds from Lemma \ref{lem:3.11}, yielding the desired result for the first term.

\begin{align}
\lVert \int_0^1 \int_0^1 A'(P^jP'^{j'}(f) + h_{\mu,\omega}(j,j',f))\mathbf{v}_1^{(j,j')}(f) d\mu d\omega \rVert_{L^{\infty}([0,1]^2)} = C_{(j,j')} 2^{-(j+j')(2\alpha + 1)}
\label{eq:54}
\end{align}

Proceeding with the second term gives:

\begin{align}
& \lVert \int_0^1 \int_0^1 A'(P^jP'^{j'}(f)) \mathbf{v}_1^{(j,j')}(f) d \mu d \omega \rVert_{L^{\infty}([0,1]^2)} = \lVert  A'(P^jP'^{j'}(f)) Q^jQ'^{j'}(f) \rVert_{L^{\infty}([0,1]^2)} \nonumber \\
& \leq \sum_{k=1}^{2^j} \sum_{k'=1}^{2^{j'}} \lVert A'(P_k^jP_{k'}'^{j'}(f)) Q_k^jQ_{k'}'^{j'}(f) \rVert_{L^{\infty}(I^j_k \times I^{j'}_{k'})} \nonumber \\
& = \sum_{k=1}^{2^j} \sum_{k'=1}^{2^{j'}} \lVert A(P_k^jP_{k'}'^{j'}(f)(x)) - A(P_k^jP_{k'}'^{j'}(f)(z)) Q_k^jQ_{k'}'^{j'}(f) \rVert_{L^{\infty}(I^j_k \times I^{j'}_{k'})} \nonumber \\
& \leq \sum_{k=1}^{2^j} \sum_{k'=1}^{2^{j'}}  \lVert C_{(j,j')}  P_k^jP_{k'}'^{j'}(f)(x) - P_k^jP_{k'}'^{j'}(f)(z) Q_k^jQ_{k'}'^{j'}(f) \rVert_{L^{\infty}(I^j_k \times I^{j'}_{k'})} \nonumber \\
& \leq \sum_{k=1}^{2^j} \sum_{k'=1}^{2^{j'}} \lVert C_{(j,j')} \eta^{(j,j')}_{(k,k')}(\chi_{(I^j_k \times I^{j'}_{k'})}(x) - \chi_{(I^j_k \times I^{j'}_{k'})}(z)) \alpha^{(j,j')}_{(k,k')}\chi_{(I^j_k \times I^{j'}_{k'})} \rVert_{L^{\infty}(I^j_k \times I^{j'}_{k'})} \nonumber \\
& = \sum_{k=1}^{2^j} \sum_{k'=1}^{2^{j'}} C_{(j,j')} 2^{-(j+j')(2\alpha + 1)}
\label{eq:55}
\end{align}

The estimate is obtained using the same arguments from the first estimate. Once the terms we measure in $L^{\infty}$ are broken down into their convolution operators, the associated expansion coefficient multiplied by the characteristic function supported on the dyadic rectangle is bounded by the measure of the dyadic rectangle to the power $\alpha$, yielding:

\begin{align}
\lVert \int_0^1 \int_0^1 A'(P^jP'^{j'}(f)) \mathbf{v}_1^{(j,j')}(f) d \mu d \omega \rVert_{L^{\infty}([0,1]^2)} = C_{(j,j')} 2^{-(j+j')(2\alpha + 1)}
\label{eq:56}
\end{align}

Estimates for the third and fourth terms in (\ref{eq:50}) are obtained with similar reasoning:

\begin{align}
& \lVert \int_0^1 \int_0^1 A''(P^jP'^{j'}(f)+ h_{\mu,\omega}(j,j',f))\mathbf{v}^{(j,j')}_2(f) d \mu d \omega \rVert_{L^{\infty}([0,1]^2)}  \nonumber \\
& \leq \lVert \sup_{\mu, \omega \in [0,1]} \int_0^1 \int_0^1 A''(P^jP'^{j'}(f)+ h_{\mu,\omega}(j,j',f))\mathbf{v}^{(j,j')}_2(f) d \mu d \omega \rVert_{L^{\infty}([0,1]^2)} \nonumber \\
& = \lVert A''(P^jP'^{j'}(f)+ Q^jQ'^{j'}(f))(P^jQ'^{j'}(f) +  Q^jQ'^{j'}(f))(Q^jP'^{j'}(f) +  Q^jQ'^{j'}(f)) \rVert_{L^{\infty}([0,1]^2)} \nonumber \\
& \leq \sum_{k=1}^{2^j} \sum_{k'=1}^{2^{j'}} \lVert A''(P_k^jP_{k'}'^{j'}(f)+ Q_k^jQ_{k'}'^{j'}(f))((P_k^jQ_{k'}'^{j'}(f) +  Q_k^jQ_{k'}'^{j'}(f))(Q_k^jP_{k'}'^{j'}(f) +  Q_k^jQ_{k'}'^{j'}(f)) \rVert_{L^{\infty}(I^j_k \times I^{j'}_{k'})} \nonumber \\
& = \sum_{k=1}^{2^j} \sum_{k'=1}^{2^{j'}}  \lVert  (A'(P_k^jP_{k'}'^{j'}(f)(x)+ Q_k^jQ_{k'}'^{j'}(f)(x)) -  A'(P_k^jP_{k'}'^{j'}(f)(z)+ Q_k^jQ_{k'}'^{j'}(f)(z))) \nonumber \\
& ((P_k^jQ_{k'}'^{j'}(f) +  Q_k^jQ_{k'}'^{j'}(f))(Q_k^jP_{k'}'^{j'}(f) +  Q_k^jQ_{k'}'^{j'}(f)) \rVert_{L^{\infty}(I^j_k \times I^{j'}_{k'})} \nonumber \\
& \leq \sum_{k=1}^{2^j} \sum_{k'=1}^{2^{j'}} \lVert   C_{(j,j')}((P_k^jP_{k'}'^{j'}(f)(x)+ Q_k^jQ_{k'}'^{j'}(f)(x)) -  (P_k^jP_{k'}'^{j'}(f)(z)+ Q_k^jQ_{k'}'^{j'}(f)(z))) \nonumber \\
& ((P_k^jQ_{k'}'^{j'}(f) +  Q_k^jQ_{k'}'^{j'}(f))(Q_k^jP_{k'}'^{j'}(f) +  Q_k^jQ_{k'}'^{j'}(f))) \rVert_{L^{\infty}(I^j_k \times I^{j'}_{k'})} \nonumber \\
& \leq  \sum_{k=1}^{2^j} \sum_{k'=1}^{2^{j'}} \lVert   C_{(j,j')} (\eta^{(j,j')}_{(k,k')} + \alpha^{(j,j')}_{(k,k')})\chi_{(I^j_k \times I^{j'}_{k'})}(x) -  (\eta^{(j,j')}_{(k,k')} + \alpha^{(j,j')}_{(k,k')})\chi_{(I^j_k \times I^{j'}_{k'})}(z) \nonumber \\
& (\hat{\omega}^{(j,j')}_{(k,k')} + \alpha^{(j,j')}_{(k,k')})\chi_{(I^j_k \times I^{j'}_{k'})}(\tilde{\omega}^{(j,j')}_{(k,k')} + \alpha^{(j,j')}_{(k,k')})\chi_{(I^j_k \times I^{j'}_{k'})} \rVert_{L^{\infty}(I^j_k \times I^{j'}_{k'})} \nonumber \\
& = C_{(j,j')}2^{-(j+j')(3\alpha + 1.5)}
\label{eq:57}
\end{align}

In the above estimate $\hat{\omega}^{(j,j')}_{(k,k')}  \chi_{(I^j_k \times I^{j'}_{k'})} \coloneq P_k^jQ'^{j'}_{k'}(f) $ is the expansion coefficient for the mixed scaling and wavelet function where the wavelet function is defined in direction $j'$ and the scaling function is defined in direction $j$ supported on the dyadic rectangle of where the tensor product between the scaling and wavelet functions are defined. We define $\tilde{\omega}^{(j,j')}_{(k,k')}  \chi_{(I^j_k \times I^{j'}_{k'})} \coloneq Q_k^jP'^{j'}_{k'}(f)$ similarly by switching the directions the wavelet and scaling functions are defined, yielding:

\begin{align}
& \lVert \int_0^1 \int_0^1 A''(P^jP'^{j'}(f)+ h_{\mu,\omega}(j,j',f))\mathbf{v}^{(j,j')}_2(f) d \mu d \omega \rVert_{L^{\infty}([0,1]^2)} \leq C_{(j,j')} 2^{-(j+j')(3\alpha + 1.5)}
\label{eq:58}
\end{align}

For the fourth term in (\ref{eq:50}) we have:

\begin{align}
& \lVert \int_0^1 \int_0^1 A''(P^j P'^{j'}(f)) \mathbf{\tilde{v}}^{(j,j')}_2(f) d\mu d\omega \rVert_{L^{\infty}([0,1]^2)} = \lVert  A''(P^j P'^{j'}(f)) \mathbf{\tilde{v}}^{(j,j')}_2 \rVert_{L^{\infty}([0,1]^2)} \nonumber \\
& = \lVert  A''(P^j P'^{j'}(f)) Q^jP'^{j'}(f)P^jQ'^{j'}(f)\rVert_{L^{\infty}([0,1]^2)} \nonumber \\
& \leq \sum_{k=1}^{2^j} \sum_{k'=1}^{2^{j'}} \lVert  A''(P_k^j P_{k'}'^{j'}(f)) Q_k^jP_{k'}'^{j'}(f)P_k^jQ_{k'}'^{j'}(f) \rVert_{L^{\infty}(I^j_k \times I^{j'}_{k'})} \nonumber \\
& =  \sum_{k=1}^{2^j} \sum_{k'=1}^{2^{j'}} \lVert A'(P_k^j P_{k'}'^{j'}(f)(x)) - A'(P_k^j P_{k'}'^{j'}(f)(z)) Q_k^jP_{k'}'^{j'}(f)P_k^jQ_{k'}'^{j'}(f)\rVert_{L^{\infty}(I^j_k \times I^{j'}_{k'})} \nonumber \\
& \leq \sum_{k=1}^{2^j} \sum_{k'=1}^{2^{j'}}  \lVert C_{(j,j')} ((P_k^j P_{k'}'^{j'}(f)(x) - P_k^j P_{k'}'^{j'}(f)(z)) Q_k^jP_{k'}'^{j'}(f)P_k^jQ_{k'}'^{j'}(f)\rVert_{L^{\infty}(I^j_k \times I^{j'}_{k'})} \nonumber \\
& = \sum_{k=1}^{2^j} \sum_{k'=1}^{2^{j'}} \lVert C_{(j,j')} \eta^{(j,j')}_{(k,k')}(\chi_{(I^j_k \times I^{j'}_{k'})}(x) - \chi_{(I^j_k \times I^{j'}_{k'})}(z)) \tilde{\omega}^{(j,j')}_{(k,k')})\hat{\omega}^{(j,j')}_{(k,k')})\chi_{(I^j_k \times I^{j'}_{k'})}\rVert_{L^{\infty}(I^j_k \times I^{j'}_{k'})} \nonumber \\
& = C_{(j,j')} 2^{-(j+j')(3\alpha + 1.5)}\nonumber \\
\label{eq:59}
\end{align}

Consequently,

\begin{align}
\lVert \int_0^1 \int_0^1 A''(P^j P'^{j'}(f)) \mathbf{\tilde{v}}^{(j,j')}_2(f) d\mu d\omega \rVert_{L^{\infty}([0,1]^2)}\leq C_{(j,j')} 2^{-(j+j')(2\alpha + 1)}
\label{eq:60}
\end{align}

Combining (\ref{eq:53}), (\ref{eq:55}), (\ref{eq:57}), and (\ref{eq:59}) yields the desired result:

\begin{align}
& \lVert \Delta_{(N,N')}(A,f) \rVert_{L^{\infty}([0,1]^2)} = \sum_{j=0}^N \sum_{j'=0}^{N'}  C_{(j,j')} (2^{-(j+j')(2\alpha + 1)} + \nonumber \\
&  2^{-(j+j')(2\alpha + 1)} +  2^{-(j+j')(2\alpha + 1)} +  2^{-(j+j')(2\alpha + 1)}) \nonumber \\
& = \sum_{j=0}^N \sum_{j'=0}^{N'}  C_{(j,j')} 2^{-(j+j')(2\alpha + 1)}
\label{eq:61}
\end{align}

\end{proof}

This leads to a proposition concerning the quantitative estimate of the residual:

\begin{proposition}
The residual has twice the regularity of $f$ such that $\Delta_{(N,N')}(A,f) \in \Lambda_{2\alpha}([0,1]^2)$ and the following estimate holds:

\begin{align}
\lVert \Delta_{(N,N')}(A,f) \rVert_{\Lambda_{2\alpha}([0,1]^2)} \leq C_A \lVert f \rVert_{\Lambda_{\alpha}([0,1]^2)}
\label{eq:62}
\end{align}
\label{prop:5.2}
\end{proposition}

\begin{proof}
From Proposition \ref{prop:5.1}, $\Delta_{(N,N')}(A,f)$ can be written as a sum over characteristic functions supported on dyadic rectangles multiplied by constant expansion coefficients on each dyadic rectangle where each coefficient satisfies the following inequality, $| \beta^{(j,j')}_{(k,k')} | \leq 2^{-(j+j')(2\alpha + 1)}$. Appealing to Lemma \ref{lem:3.13} completes the proof. Since $\Delta_{(N,N')}(A,f) \in \Lambda_{2\alpha}([0,1]^2)$, the second part of the theorem can be proved using the norm from Definition \ref{defn:3.3}. Compute 

\begin{align}
& \lVert \Delta_{(N,N')}(A,f) \rVert_{\Lambda_{2\alpha}([0,1]^2)} = \sup_{j,k,j',k'} \frac{| <\Delta_{(N,N')}(A,f)(x,x'), \psi^j_k(x) \times \psi^{j'}_{k'}(x')> |}{2^{-(j+j')(2\alpha + 1)}} \nonumber \\
& = C_{(N,N')}^{(2\alpha + 1)}
\label{eq:63}
\end{align}

from Proposition \ref{prop:5.1} and 

\begin{align}
& \lVert f \rVert_{\Lambda_{\alpha}([0,1]^2)} = \sup_{j,k,j',k'} \frac{| <f(x,x'), \psi^j_k(x) \times \psi^{j'}_{k'}(x')> |}{2^{-(j+j')(\alpha + \frac{1}{2})}} \nonumber \\
& = C_{(N,N')}^{(\alpha + \frac{1}{2})}
\label{eq:64}
\end{align}

from Lemma \ref{lem:3.12}. Since $C_{(N,N')}$ is a constant related to the size of the dyadic rectangle, $C_{(N,N')} \leq 1$, producing:

\begin{align}
\lVert \Delta_{(N,N')}(A,f) \rVert_{\Lambda_{2\alpha}([0,1]^2)} \leq C_{A} \lVert f \rVert_{\Lambda_{\alpha}([0,1]^2)} 
\label{eq:65}
\end{align}

completing the second part of the proof.

\end{proof}

\section{Computational Example}

\begin{comment}
\begin{enumerate}
    \item Singularity in complex plane
    \item Visualization for different $\alpha,(N,N')$ (i.e. regularity/scales taken in approximation)
    \item Plot numerically substantiating quantitative estimate on residual i.e. 2$\alpha$-norm 
\end{enumerate}
\end{comment}

As a toy example, we consider a function supported on the complex plane with a ring singularity to illustrate the utility of the multiscale tensor paraproduct, $\Pi_{(A',A'')}^{(N,N')}$. Let $f_{\alpha}(z) \in \Lambda_{\alpha}([0,1] \times i[0,1]), \in 0 < \alpha < \frac{1}{2}$ such that

\begin{align}
f_{\alpha}(z) = 
\begin{cases}
(0.3 - |z|)^{\alpha} &  \text{if } |z| <  0.3 \\
(1 - \frac{0.3}{|z|})^{\alpha} & \text{if } \frac{0.3}{|z|} < 1
\end{cases}
\end{align}

Let $z = x + iy$ with $x,iy \in [0,1]$. We sample $N = 512$ equispaced points between $[0,1]$ on both the real and imaginary axis, and perform the multiscale paraproduct decomposition for $A(f) = e^{-0.2f_{\alpha}(z)}$ for $\alpha=4 \times 10^{-3}, 4 \times 10^{-2}, 4 \times 10^{-1}$ up to two settings of fixed scales, $(N=4,N'=4)$ and $(N=6,N'=6)$ and visualize $f , A(f) \coloneq e^{-0.2f(z)}$, and the approximation, $\tilde{A}_{(N,N')}(f)$ in Figure \ref{fig:6.1}

\begin{figure}[H]
    \centering
    \includegraphics[width=1.2\linewidth]{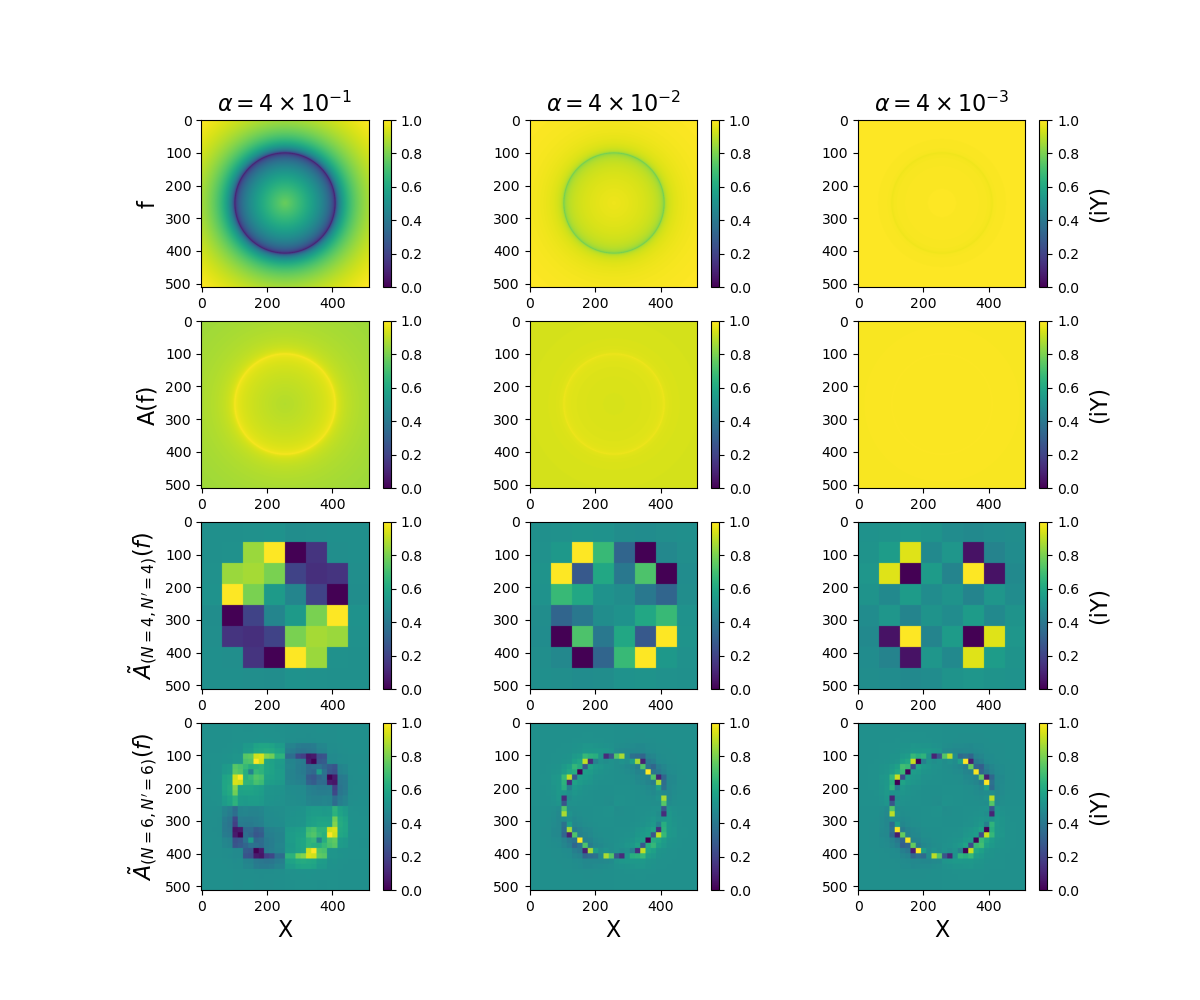}
    \caption{ Visualization of the paraproduct decomposition of $A(f) = e^{- 0.2 f(z) }$ for fixed scales $N,N' = 4$ (third row) and $N,N' = 6$ (last row) for $\alpha = 4 \times 10^{-1} , 4 \times 10^{-2}, 4 \times 10^{-3}$. Each column is a different $\alpha$ level and row 1 is the original function, $f$, row 2 is the smooth nonlinear composition, $A(f)$, and rows 3 and 4 are the approximations, $\tilde{A}_{(N=4,N'=4)}(f)$ and $\tilde{A}_{(N=6,N'=6)}(f)$, respectively. $X$ and $iY$ are the labels for the real and complex axes, respectively.}
    \label{fig:6.1}
\end{figure}

In Figure \ref{fig:6.1} the ring singularity is blurred by $A(f)$ as shown by the second row. Perturbing the $\alpha$ level also has a smoothing effect as seen by the columns. In particular, as $\alpha \to 0 ,f$ approaches a constant. The approximations $\tilde{A}_{(N=4,N'=4)}(f)$ and $\tilde{A}_{(N=6,N'=6)}(f)$ both corroborate our theoretical findings as some qualitative notion of the ring singularity is extracted in both instances; however, $\tilde{A}_{(N=6,N'=6)}(f)$ generally does a better job which is expected since additional scales are incorporated into the approximation. Of note is a peculiar effect that occurs as $\alpha \to 0$ for both approximations. For $\tilde{A}_{(N=4,N'=4)}(f)$ the approximations worsen as $\alpha \to 0$ as expected; this is because at coarse scales, the smoothing of $f$ incurred by decreasing $\alpha$ precludes the ring singularity from being discovered. The reverse relation holds for $\tilde{A}_{(N=6,N'=6)}(f)$. As $\alpha \to 0$, the approximation gets better; this is likely because at fine scales, the coarse scale tensor basis functions remove most of the smoothing from taking $\alpha \to 0$ while the finest scale tensor basis functions in the approximation only extract the ring singularity.  

\begin{comment}
\textbf{To Do}
\begin{enumerate}
    \item $N,N' = 4$
    \item $N,N' = 6$
    \item $N,N' = 8$
    \item For the final fixed scales $N,N' = 6$ show a plot with exponential decay in the scaling parameters for each of the $\alpha$ levels.
    \item Bonus: Repeat for $N,N' = 4 = 6$ For each of the fixed scales above, show a plot with exponential decay 
\end{enumerate}
\end{comment}

\bibliographystyle{amsplain}
\bibliography{main}

\end{document}